\documentclass[OJMO,NoBabel,manuscript]{cedram}

\usepackage[sort]{cite}
\usepackage[all]{xypic}
\usepackage{graphicx} 
\usepackage{bm}
\usepackage{etoolbox}
\usepackage{algorithmic}
\usepackage[ruled,vlined,linesnumbered,inoutnumbered]{algorithm2e}
\usepackage{caption}
\usepackage{subcaption}

\usepackage{latex_styles_common}

\usepackage[nameinlink]{cleveref}
\crefname{equation}{}{}
\Crefname{figure}{Figure}{Figures}
\crefname{figure}{Figure}{Figures}
\crefname{example}{Example}{Example}
\crefname{theorem}{Theorem}{Theorem}
\crefname{corollary}{Corollary}{Corollary}
\crefname{lemma}{Lemma}{Lemma}
\crefname{proposition}{Proposition}{Proposition}
\crefname{assumption}{Assumption}{Assumption}
\crefname{algorithm}{Algorithm}{Algorithm}

\crefformat{section}{\S#2#1#3} 
\crefformat{subsection}{\S#2#1#3}
\crefformat{subsubsection}{\S#2#1#3}

\usepackage{enumitem}
\newlist{thmenum}{enumerate}{1} 
\setlist[thmenum]{label=\alph*{\rm)}, ref=\thetheorem(\alph*)} 
\crefalias{thmenumi}{theorem}

\usepackage{amsthm,thmtools}
\declaretheorem[name=Theorem]{theorem}
\declaretheorem[name=Definition,style=definition,numberlike=theorem]{definition}
\declaretheorem[name=Example,style=definition,numberlike=theorem,qed=\qedsymbol]{example}
\declaretheorem[name=Proposition,numberlike=theorem]{proposition}
\declaretheorem[name=Corollary,numberlike=theorem]{corollary}
\declaretheorem[name=Assumption,numberlike=theorem]{assumption}
\declaretheorem[name=Lemma,numberlike=theorem]{lemma}

\definecolor{blu}{rgb}{0,0,1}

\definecolor{gre}{rgb}{0,.5,0}

\definecolor{red}{rgb}{1,0,0}

\DeclareMathOperator{\EssCone}{\sf EssCone}
\DeclareMathOperator{\EssSpan}{\sf EssSpan}

\input{macros}

\newcommand{\Face}[1]{\Fscr_{\scriptscriptstyle#1}}

\newcommand{\gauge}{\gamma}
\newcommand{\As}{_{\scriptscriptstyle\Ascr}}

\DeclareMathOperator{\suppa}{\Sscr\As}

\newcommand{\Probi}{(P$_i$)}
\newcommand{\Drobi}{(D$_i$)}
\newcommand{\opa}[1]{\|#1\|\As}

\renewcommand{\T}{^\intercal}

\title{Cardinality-constrained structured data-fitting problems}

\author{\firstname{Zhenan} \lastname{Fan} }
\address{The University of British Columbia, Canada}
\email{zhenanf@cs.ubc.ca}

\author{\firstname{Huang} \lastname{Fang}}
\address{The University of British Columbia, Canada}
\email{hgfang@cs.ubc.ca}
\author{\firstname{Michael} \middlename{P.} \lastname{Friedlander}}
\address{The University of British Columbia, Canada}
\email{michael.friedlander@ubc.ca}

\thanks{Zhenan Fan and Huang Fang contributed equally to this paper.}

\keywords{convex analysis, sparse optimization, low-rank optimization, primal-retrieval}
  
\begin{abstract}%
  A memory-efficient framework is described for the cardinality-constrained structured data-fitting problem.  Dual-based atom-identification rules are proposed that reveal the structure of the optimal primal solution from near-optimal dual solutions. These rules allow for a simple and computationally cheap algorithm for translating any feasible dual solution to a primal solution that satisfies the cardinality constraint. Rigorous guarantees are provided for obtaining a near-optimal primal solution given any dual-based method that generates dual iterates converging to an optimal dual solution. Numerical experiments on real-world datasets support confirm the analysis and demonstrate the efficiency of the proposed approach.
\end{abstract}

\begin{document}

\maketitle


\section{Introduction}\label{sec:intro}

Consider the problem of fitting a model to data by building up model parameters as the superposition of a limited set atomic elements taken from a given dictionary. Versions of this cardinality-constrained problem appear in a number of statistical-learning applications in machine learning~\cite{tibshirani1996regression,yul06,Meinshausen06,aep08}, data mining, and signal processing~\cite{candes:2013}. In these applications, common atomic dictionaries include the set of one-hot vectors or matrices (i.e., vectors and matrices that contain only a single element) and rank-1 unit matrices. The elements of the given dictionary encode a notion of parsimony in the definition of the model parameters.

The cardinality-constrained formulation that we consider aims to
\begin{equation} \label{eq:main_prob} \tag{P}
    \find x\in\Re^n \sut f(b - Mx) \leq \alpha \tand \card_\Ascr(x) \leq k,
\end{equation}
where $f:\Re^n\to\Re$ is an $L$-smooth and convex function, $M: \Re^n \to \Re^m$ is a linear operator with $m < n$, $b \in \Re^m$ is the observation vector, and $\Ascr \subseteq \Re^n$ is the atomic dictionary. The cardinality function 
\begin{equation}\label{eq-card-fcn}
    \card_\Ascr(x) \coloneqq \inf\Set{\mathbf{nnz}(c) | x = \textstyle\sum\limits_{a \in \Ascr} c_a a, \enspace c_a \geq 0}
\end{equation}
measures the complexity of $x$ with respect to the atomic set $\Ascr$.  The loss term $f(b - Mx)$ measures the quality of the fit. When $\Ascr = \set{ \pm e_1,\pm e_2, \ldots,\pm e_n }$ is the set of signed canonical unit vectors, the function $\card_\Ascr(x)$ simply counts the number of nonzero elements in $x$. Typically $k\ll n$, which indicates that we seek a feasible model parameter $x$ that has an efficient representation in terms of $k$ atoms from the set $\Ascr$.

For the application areas that we target, the two characteristics of this feasibility problem that pose the biggest challenge to efficient implementation are the combinatorial nature of the cardinality constraint, and the high-dimensionality of the parameter space. To address the combinatorial challenge, we follow van den Berg and Friedlander~\cite{berg2008probing,berg2011sparse} and Chandrasekaran et al.~\cite{chandrasekaran2012convex}, and use the convex gauge function
\begin{equation} \label{eq:min_sum}
    \gauge\As(x) = \inf\Set{
    \textstyle\sum\limits_{a \in \Ascr} c_a ~\bigg\vert~ x = \sum\limits_{a \in \Ascr} c_a a, \ c_a \geq 0}
\end{equation}
as a tractable proxy for the cardinality function~\eqref{eq-card-fcn}; see \cref{sec:preliminaries}. In tandem with the convexity of the loss function, this function allows us to formulate three alternative relaxed convex optimization problems that, under certain conditions, have approximate solutions satisfying the feasibility problem; see problems~\eqref{prob:primal1},~\eqref{prob:primal2}, and~\eqref{prob:primal3} in~\cref{sec:probSetting}.

The high-dimensionality of the parameter space, however, may imply that it's inefficient---and maybe even practically impossible---to solve these convex relaxations because it's infeasible to store directly the approximations to a feasible solution $x$. Instead, we wish to develop methods that leverage the efficient representation that feasible solutions have in terms of the atoms in the dictionary $\Ascr$. For example, consider the case in which the dictionary is the set of symmetric ${n \times n}$ rank-one matrices, and $M$ is the trace linear operator that maps these matrices into $m$-vectors. Any method that iterates directly on the parameters $x$ requires $\BigOh(n^2+m)$ storage for the iterates and the data. An alternative is the widely-used conditional gradient method \cite{frank1956algorithm}, which requires $\BigOh( n t+m )$ storage after $t$ iterations~\cite{jaggi2013revisiting}, but also often requires a substantial number of iterations $t$ to converge. Instead of storing $x$ directly, we apply a dual method to one of the convex relaxations~\eqref{prob:primal1}, \eqref{prob:primal2}, and~\eqref{prob:primal3} (defined below); these dual methods typically require only $\BigOh(m)$ storage, and still allow us to collect information on which atoms in $\Ascr$ participate in the construction of a feasible $x$. One of the aims of this paper is to describe how to collect and use this information.

\subsection{Approach}

We propose a unified algorithm-agnostic strategy that uses any sequence of improving dual solutions to one of the convex relaxations. This dual sequence identifies an essential subset of atoms in $\Ascr$ needed to construct an $\epsilon$-infeasible solution $x$ that satisfies the conditions
\begin{equation*}
    f(b - Mx) \leq \alpha + \epsilon \quad\mbox{and}\quad \card(\Ascr; x) \leq k
\end{equation*}
for any positive tolerance $\epsilon$. These \emph{atomic-identification} rules, described in \cref{sec:atom_iden}, derive from the polar-alignment property and apply to arbitrary dictionaries $\Ascr$ \cite{fan2019alignment}. These atom-identification rules generalize earlier approaches described by El~Ghaoui~\cite{Ghaoui12} and Hare and Lewis~\cite{hare2004identifying}. Once an essential subset of $k$ atoms is identified, an $\epsilon$-feasible solution $x$ can be computed by optimizing over all positive linear combinations of this subset. This relatively small $k$-variable problem can often be solved efficiently.

We prove that when the atomic dictionary is polyhedral, we can set $\epsilon$ to zero and still identify in polynomial time a set of feasible atoms; see \cref{coro:polyhedral}. When the atomic dictionary is spectrahedral, we prove that an $\epsilon$-feasible set of atoms can be identified also in polynomial time; see \cref{coro:spectral}. 

We demonstrate via numerical experiments on real-world datasets that this approach is effective in practice.

There are three important elements in our primal-retrieval algorithm. The first element is an atom-identifier function $\EssCone_{\Ascr, k}$ that maps $M^*y$, where $y$ is any feasible dual variable, to a cone generated by $k$ atoms that are \emph{essential}. These atoms have the property that 
\[
  \EssCone_{\Ascr, k}(M^*y) \subseteq \set{x | \card(A; x) \leq k}.
\]
The explicit definition of the essential cone depends on the particular atomic set $\Ascr$. In \cref{sec:primal-retrieval}, we make it explicit for atomic sets that are polyhedral (\cref{sec:polyhedral}) and spectral (\cref{sec:spectral}).

The second element is an arbitrary function $\texttt{oracle}_{f,\Ascr,M,b}$ (such as an appropriate first-order iterative method) that generates dual iterates $y^{(t)}$ converging to the optimal dual variable $y^*$ of any of the dual problem \Drobi. It's this oracle that generates the dual estiamtes subsequently used by $\EssCone_{\Ascr,k}$.

The third algorithmic component is the reduced convex optimization problem
\begin{equation} \label{eq:primal_recover} \tag{PR}
    x^{(t)} \in \argmin{x}\set{ f(b - Mx) | x \in \EssCone_{\Ascr, k}(M^*y^{(t)})},
\end{equation}
which at each iteration constructs a primal estimate $x^{(t)}$ using the atoms identified through the dual estimate $y^{(t)}$.  The detailed algorithm is shown in \cref{alg:primal_recover}. Note that our primal-retrieval strategy doesn't aim to recover the optimal solutions to~\eqref{prob:primal1},~\eqref{prob:primal2} or~\eqref{prob:primal3}. These problems serve as a guidance for our atom-identification rule. The final output of \cref{alg:primal_recover} maybe different from the optimal solutions to~\eqref{prob:primal1}, \eqref{prob:primal2}, or~\eqref{prob:primal3}. 

\RestyleAlgo{ruled}
\SetKwInput{KwInput}{Input}
\SetKwInput{KwReturn}{Return}
\begin{algorithm}[t]
\DontPrintSemicolon{}
\caption{primal-retrieval algorithm}\label{alg:primal_recover}
\KwInput{data-fitting constraint $\alpha$, cardinality constraint $k$, atomic set $\Ascr$, loss function $f$, linear operator $M$, observation $b$, and tolerance $\epsilon\ge0$
    }
Initialize dual feasible vector $y^{(0)}$\; 
 \For{t = 1, 2, \dots}{
    $y^{(t)} \leftarrow \texttt{oracle}_{f,\Ascr,M,b}(y^{(t-1)})$\;
    $x^{(t)} \leftarrow$ solution to~\eqref{eq:primal_recover}\;
    \If{$f(b - Mx^{(t)}) \leq \alpha + \epsilon$}{break}
 }
 \KwReturn{ $x^{(t)}$ }
\end{algorithm}

\section{Related work}

Many recent approaches for atomic-sparse optimization problems are based on algorithms~\cite{fan2019bundle,DingYCTU21}. These methods, however, still need to retrieve at some point a primal solution $x$, which may require a prohibitive amount of memory for its storage. A widely used heuristic applies the truncated singular value decomposition to obtain low-rank solutions, but this heuristic is unreliable in minimizing the model misfit~\cite[Algorithm~6.4]{fan2019alignment}. Memory-efficient atomic-sparse optimization thus requires efficient and reliable methods to retrieve an atomic-sparse primal solution.

Dual approaches for nuclear-norm or trace-norm regularized problems are attractive because they enjoy optimal storage, which means that they have space complexity $\BigOh(m)$ instead of $\BigOh( n^2 )$ \cite{DingYCTU21}. For example, the bundle method for solving the Lagrangian dual formulation of semi-definite programming~\cite{helmberg2000spectral}, and the gauge dual formulation of general atomic sparse optimization problem~\cite{fan2019bundle}, exhibit promising results in practice.

A related line of research uses memory-efficient primal-based algorithms based on hard-thresholding. Some examples include gradient hard-thresholding~\cite{YuanLZ17}, periodical hard-thresholding~\cite{Allen-ZhuHHL17}, and many proximal-gradient or ADMM-based hard-thresholding algorithms~\cite{mazumder2010spectral,Lin11,hsieh2014nuclear}. These approaches are primal-based and tangential to to our purposes. We do not include them in our discussion. 

The theoretical analysis of our primal-retrieval approach is related to optimal atom identification~\cite{BurM88,hare2004identifying,hare2011identifying}, and especially to recently developed safe-screening rules for various kind of sparse optimization problems~\cite{Ghaoui12,wang2013lasso,liu2014safe,WangZLWY14,Raj2015ScreeningRF,BonnefoyERG15,XiangWR17,NdiayeFGS17,ZhangHLYCHW17,kuang2017screening,Atamtrk2020SafeSR,Bao20}. One of our main results, given by \Cref{thm:p0}, generalizes the gap safe-screening rule developed by Ndiaye et al.~\cite{ndiaye2016gap} to general atomic-sparse problems and to more general problem formulations. Some of the techniques used in our analysis are related to the facial reduction strategy from Krislock and Wolcowicz~\cite{krislock2010explicit}.

\section{Preliminaries}\label{sec:preliminaries}

We introduce in this section the main tools of convex analysis used to understand atomic sparsity.

The gauge function~\eqref{eq:min_sum} is always convex, nonnegative, and positively homogeneous. This function is finite only at points contained within the cone
\begin{equation*}
  \cone(\Ascr) \coloneqq \Set{x = \sum\limits_{a \in \Ascr} c_a a ~\bigg\vert~ c_a \geq 0}
\end{equation*}
generated from the elements of the set $\Ascr$.  The gauge is not necessarily a norm because it may not be symmetric (unless $\Ascr$ is centrosymmetric), may vanish at points other than the origin, and may not be finite valued (unless $\Ascr$ contains the origin in the interior of its convex hull). Throughout, we make the blanket assumption that the dictionary $\Ascr\subseteq\Re^n$ is compact, and that the origin is contained in its convex hull. The assumption on the origin ensures that the gauge function is continuous. The compactness assumption isn't strictly necessary for many of our conclusions, but does considerably simplify the analysis. The set $\Ascr$ may be nonconvex, which is the case, for example, if it consists of a discrete set of two or more items. 

The definition of the gauge function makes explicit this function's role as a convex penalty for atomic sparsity. The \emph{atomic support} of a vector $x$ to be the collection of atoms $a\in\Ascr$ that contribute positively to the conic decomposition implied by the value $\gamma\As(x)$~\cite[Definition~2.1]{fan2019alignment}.

\begin{definition}[Atomic support]
  A subset of atoms $\suppa(x)\subset\Ascr$ is a \emph{support set} for $x$ with respect to $\Ascr$ if every atom $a\in\suppa(x)$ satisfies
  \begin{equation*}
    \gauge\As(x) = \sum_{\mathclap{a\in\suppa(x)}} c_a,
    \qquad x = \sum_{\mathclap{a\in\suppa(x)}} c_a a,
    \text{and} c_a > 0\quad \forall a\in\suppa(x).
  \end{equation*}
\end{definition}
For example, the support set $\suppa(x)$ for the atomic set of 1-hot unit vectors $\Ascr = \set{ e_i \mid i = 1,2,\ldots,n }$ coincides with the nonzero elements of $x$ with the corresponding sign. The support function to the set $\Ascr$ is given by $\sigma\As(z):=\sup_{a\in\Ascr}\ip{a}{z}$. Because $\Ascr$ is compact, every direction $d\in\Re^n$ generates a supporting hyperplane to the convex hull of $\Ascr$. The atoms contained in that supporting hyperplane are said the be \emph{exposed} by $d$. The following definition also includes the notion of atoms that are approximately exposed.

\begin{definition}[Exposed and $\epsilon$-exposed atoms]\label{def:face} 
  The exposed atoms and $\epsilon$-exposed atoms, respectively, of a set $\Ascr\subseteq\Re^n$ in the direction $z \in \Re^n$ are defined by the sets
  \begin{align*}
    \Escr\As(z) \coloneqq \set{a \in \Ascr \mid \ip{a}{z} = \sigma\As(z)} \tand \Escr\As(z, \epsilon) \coloneqq \set{a \in \Ascr \mid \ip{a}{z} \geq \sigma\As(z) -
      \epsilon},
  \end{align*}
  where $\sigma\As(z) \coloneqq \sup_{a \in \Ascr} \ip{a}{z}$ is the support function
  with respect to $\Ascr$.
\end{definition}
When $\epsilon = 0$, the $\epsilon$-exposed atoms coincide with the exposed atoms. 

We list in \Cref{tab:common-atoms} commonly used atomic sets, their
corresponding gauge and support functions, and atomic supports.

\begin{table*}
\centering
\begin{tabular}{@{}lcccc@{}}  
\toprule 
\small
Atomic sparsity & $\Ascr$  & 
  $\gauge\As(x)$ & $\suppa(x)$ & $\sigma\As(z)$ \\
\midrule
non-negative & $\cone(\set{\bm{e}_1,\ldots,\bm{e}_n})$  &
  $\delta_{\geq 0}$ & $\cone(\set{\bm{e}_i \mid x_i > 0})$ & $\delta_{\leq 0}$\\

element-wise & $\set{ \pm \bm{e}_1,\ldots,\pm \bm{e}_n }$  & 
  $\|\cdot\|_1$ & $\set{\sign(x_i)\bm{e}_i \mid x_i \neq 0}$ & $\|\cdot\|_\infty$\\ 

low rank & $\set{uv^T \mid \|u\|_2=\|v\|_2 = 1}$  & 
  $\|\cdot\|_*$  & singular vectors of $x$ & $\|\cdot\|_2$\\

PSD \& low rank & $\set{uu^T \mid \|u\|_2 = 1}$  & 
  $\trace + \delta_{\succeq 0}$  & eigenvectors of $x$ & $\max\set{\lambda_{\max}, 0}$\\ 
\bottomrule 
\end{tabular}
\caption[Commonly used sets of atoms and their gauge and support function
representations]{Commonly used atomic sets and the corresponding gauge and
  support functions~\cite{fan2019bundle}. The indicator function
  $\delta_{\Cscr}(x)$ is zero if $x$ is in the set $\Cscr$ and $+\infty$ otherwise. 
  The commonly used group-norm is also an atomic norm~\cite[Example~4.7]{fan2019alignment}. The functions $\|X\|_*$ and $\|X\|_2$, respectively, correspond to the nuclear norm (sum of singular values) and spectral norm (maximum singular value) of a matrix $X$.
\label{tab:common-atoms}}  
\end{table*}

\section{Atomic-sparse optimization}
\label{sec:probSetting}

We introduce in this section convex relaxations to the structured data-fitting problem~\eqref{eq:main_prob}. In particular, we consider the following three related gauge-regularized optimization
problems:
\begin{align} 
  & \minimize{x}\enspace p_1(x) := f(b - Mx) + \lambda\gauge\As(x), \label{prob:primal1} \tag{P$_1$} \\
  & \minimize{x}\enspace p_2(x) := f(b - Mx) \enspace\text{subject to}\enspace \gauge\As(x) \leq \tau, \label{prob:primal2} \tag{P$_2$} \\
  & \minimize{x}\enspace p_3(x) := \gauge\As(x)  \enspace\text{subject to}\enspace f(b - Mx) \leq \alpha. \label{prob:primal3} \tag{P$_3$} 
\end{align}
It's well known that under mild conditions, these three formulations are equivalent for appropriate choices of the positive parameters $\lambda, \tau$, and $\alpha$~\cite{FrieTsen:2006}. Practitioners often prefer one of these formulations depending on their application. For example, tasks related to machine learning, including feature selection and recommender systems, typically feature one of the first two formulations~\cite{tibshirani1996regression,yul06,Meinshausen06}.  On the other hand, applications in signal processing and related fields, such as as compressed sensing and phase retrieval, often use the third formulation~\cite{berg2008probing,candes:2013}. 

Our primal-retrieval strategy relies on the hypothesis that the atomic-sparse optimization problems~\eqref{prob:primal1},~\eqref{prob:primal2} and~\eqref{prob:primal3} are reasonable convex relaxations to the structured data-fitting problem~\eqref{eq:main_prob}, in the sense that the corresponding optimal solutions are feasible for~\eqref{eq:main_prob}. We formalize this in the following assumption.

\begin{assumption}\label{ass:blanket}
   Let $x^*$ denote an optimal solution to \Probi, $i=1,2,3$. Then $x^*$ is feasible for \eqref{eq:main_prob}, i.e.,
   \[f(b - Mx^*) \leq \alpha \tand |\suppa(x^*)| \leq k.\]
\end{assumption}

As described in \cref{sec:intro}, the Fenchel-Rockafellar duals for these problems have typically smaller space complexity. These dual problems can be formulated as
\begin{align} 
  & \minimize{y}\enspace d_1(y) := f^*(y) - \ip{b}{y} \enspace\text{subject to}\enspace \sigma\As(M^*y) \leq \lambda, \label{prob:dual1} \tag{D$_1$}\\
  & \minimize{y}\enspace d_2(y) := f^*(y) - \ip{b}{y} + \tau\sigma\As(M^*y), \label{prob:dual2} \tag{D$_2$}\\
  & \minimize{y}\enspace \inf_{\beta > 0} \,d_3(y, \beta)
    := \beta \left(
      f^*\left( y/\beta \right) + \alpha
    \right) - \ip{b}{y} \enspace\text{subject to}\enspace \sigma\As(M^* y) \leq 1, \label{prob:dual3} \tag{D$_3$}
\end{align}
where $f^*(y) = \sup_w \enspace \ip{y}{w} - f(w)$ is the convex conjugate
function of $f$, and $M^*:\Re^m \to \Re^n$ is the adjoint operator of $M$, which satisfies $\ip{Mx}{y} = \ip{x}{M^*y}$ for all $x\in\Re^n$ and $y\in\Re^m$. The derivation of these dual problems can be found in \cref{app:duals}. 


\section{Atom identification} \label{sec:atom_iden}

We demonstrate in this section how an optimal dual solution can be used to identify essential atoms that participate in the primal solutions. In order to develop atomic-identification rules that apply to arbitrary atomic sets $\Ascr\subseteq\Re^n$---even those that are uncountably infinite---we require generalized notions of active constraint sets. In linear programming, for example, the simplex multipliers give information about the optimal primal support. By analogy, our atomic-identification rules give information about the essential atoms that participate in the support of the primal optimal solutions. In addition, we extend the identification rules to approximate the essential atoms from approximate dual solutions. 

We build on the following
result, due to Fan et al.~\cite[Proposition~4.5 and Theorem~5.1]{fan2019alignment}.
\begin{theorem}[Atom identification]\label{thm:opt_supp_id} Let
  $x^*$ and $ y^*$ be optimal primal-dual solutions for problems \Probi~and \Drobi, with
  $i=1,2,3$. Then
  \begin{equation}\label{eq:1exact}
    \suppa(x^*) \subseteq \Escr\As(M^*y^*).
  \end{equation}
\end{theorem}

The following theorem generalizes this result to show similar atomic support identification properties that also apply to approximate dual solutions. In particular, given a feasible dual variable $y$ close to $y^*$, the support of $x^*$ is contained in the set of $\epsilon$-exposed atoms that includes $\Escr\As(M^*y^*)$. 
\begin{theorem}[Generalized atom identification]\label{thm:p0}
  Let $x_i$ and $y_i$ be feasible primal and dual vectors, respectively for
  problems \Probi\ and \Drobi, with $i = 1, 2, 3$. Then
  \begin{equation}\label{eq:1} 
    \suppa(x_i^*) \subseteq \Escr\As(M^*y_i, \epsilon_i),
  \end{equation}
  where each $\epsilon_i$ is defined for problem $i$ by
  \begin{thmenum} \setlength\itemsep{-0.0em}
  \item \label{thm:p1} $\epsilon_1 = \opa{M}\sqrt{2L \left( d_1(y_1) - d_1(y_1^*) \right)}$,

  \item \label{thm:p2} $\epsilon_2 = 2\opa{M}\sqrt{2L \left( d_2(y_2) - d_2(y_2^*) \right)}$,
    
  \item \label{thm:p3} $\epsilon_3 = 2\opa{M}\sqrt{2\bar{\beta}L (
        \max\set{d_3(y_3, \underline{\beta}), d_3(y_3, \bar{\beta})} - d_3( y_3^*, \beta^*) ) }$, 
  \end{thmenum}
  where $\underline{\beta}$ and $\bar{\beta}$ are positive lower and
  upper bounds, respectively, for $\beta^*$, and $\|M\|\As:=\max_{a\in\Ascr}\|Ma\|_2$ is the induced atomic operator norm.
\end{theorem}

\Cref{thm:p0} asserts that the underlying atomic support of $x_i^*$ is contained in the set of the $\epsilon$-exposed atoms of $M^* y_i$.  Moreover, when $y_i\to y_i^*$ (and, for problem~\eqref{prob:dual3}, the bounds $\underline\beta\to\beta^*$ and $\bar\beta\to\beta^*$), each $\epsilon_i\to0$, and thus~\eqref{eq:1} implies that we have a tighter containment for the optimal atomic support.  The proofs for parts (a) and (b) of \cref{thm:p0} depend on the strong convexity of $f^*$, which is implied by the Lipschitz smoothness of $f$~\cite[Theorem 4.2.1]{hiriart-urruty01}. This convenient property, however, is not available for part (c) because the dual objective of \eqref{prob:dual3} is the perspective map of $f^* + \alpha$, which is not strongly convex~\cite{aravkin2018foundations}. We resolve this technical difficulty by instead imposing the additional assumption that bounds are available on the dual optimal variable $\beta^*$. Appendix~\ref{app:bounds} describes how to obtain these bounds during the course of the level-set method developed by Aravkin et al.~\cite{aravkin2016levelset}.

The gap safe-screening rule developed by Ndiaye et al.~\cite{NdiayeFGS17} is a special case of \cref{thm:p0} that applies only to~\eqref{prob:primal1} for the particular case in which $\gamma\As$ is the one-norm. We also note that \cref{thm:p0} states a tighter bound on $\epsilon_1$ that relies on the optimal primal value $d_1(y_1^*)$ rather than the primal value at an arbitrary primal-feasible point.

\section{Primal retrieval} \label{sec:primal-retrieval}

\Cref{thm:p0} serves mainly as a technical tool for error bound analysis, in particular because it's impractical to compute or approximate $\epsilon_i$. However, the inclusions~\eqref{eq:1exact} and~\eqref{eq:1}, respectively, of \cref{thm:opt_supp_id,thm:p0} motivate us to define an atom-identifier function $\EssCone_{\Ascr, k}$ that depends on the dual variable $y$ and satisfies the inclusions 
\[\cone(\Escr\As(M^*y)) \subseteq \EssCone_{\Ascr, k}(M^*y) \subseteq \cone(\Escr\As(M^*y, \epsilon)).\]

The next two sections demonstrate how to construct such a function for polyhedral and spectral atomic sets, which are two important examples that appear frequently in practice. With this function we can thus implement the primal-recovery problem required by Step~5 of \cref{alg:primal_recover}. Moreover, we show how to use the error bounds of \cref{thm:p0} to analyze the atomic-identification properties of the resulting algorithm.

\subsection{Primal-retrieval for polyhedral atomic sets} \label{sec:polyhedral}

We formalize in this section a definition for the function $\EssCone_{\Ascr, k}$ for the case in which $\Ascr$ is a finite set of vectors, which implies that the convex hull is  polyhedral. Given a feasible dual vector $y$, consider the top-$k$ atoms in $\Ascr$ with respect to the inner product with the vector $M^*y$:
\begin{equation} \label{eq:top-k}
    \Ascr_k := \{a_1, \dots, a_k\} \subseteq \Ascr \text{such that} \ip{M^*y}{a_i} \geq \ip{M^*y}{a} \enspace\forall i \in [k] ~\mbox{and}~ a \in \Ascr\setminus\{a_1, \dots, a_k\}.
\end{equation}
Note that there may be many sets of $k$ atoms that satisfy this property. We then construct the cone of essential atoms as the convex conic hull generated from this set of top-$k$ atoms:
\begin{equation*}
  \EssCone_{\Ascr, k}(M^*y) \coloneqq \cone\Ascr_k.
\end{equation*}
Thus, the primal-retrieval computation in step~5 of \cref{alg:primal_recover} is given by 
\begin{equation*}
\hat x = \sum_{i=1}^k \hat c_ia_i, 
\end{equation*}
where
\begin{equation} \label{eq-top-k-recover}
  \hat c \in \argmin{c \in \Re^k_+} f_k(c),
  \text{with} f_k(c) := f\left( b - M\sum_{i=1}^k c_ia_i \right),
\end{equation}
is the $k$-vector of coefficients obtained by minimizing the reduced objective over a $k$-dimensional polyhedron defined by the top-k atoms.

\begin{example}[Sparse vector recovery] \label{ex:bpdn}

We consider the problem of recovering a sparse vector $x^\natural$ from noisy observations $b:=Mx^\natural + \eta$, where $M:\Re^n \to \Re^m$ is a given measurement matrix and $\eta \in \Re^m$ is standard Gaussian noise. For some expected noise level $\alpha > 0$, the sparse recovery problem can be thus be expressed as 
\begin{equation*}
    \find x\in\Re^n \sut \|b - Mx\|_2 \leq \alpha \tand \texttt{nnz}(x) \leq k,
\end{equation*}
which corresponds to~\eqref{eq:main_prob} with $f=\|\cdot\|_2$ and with the atomic set
\begin{equation}\label{eq-signed-e}
  \Ascr = \{\pm e_1, \dots, \pm e_n\},
\end{equation}
where each $e_i$ is the $i$th canonical unit vector. The basis pursuit denoising (BPDN) approach approximates this problem by replacing the cardinality constraint with an optimization problem that minimizes the 1-norm of the solution:
\begin{equation} \label{eq:bpdn}
  \minimize{x\in\Re^n} \|x\|_1  \enspace\mbox{subject to}\enspace \|Mx - b\|_2 \leq \alpha;
\end{equation}
see Chen et al.~\cite{cds98}. This convex relaxation corresponds to problem \eqref{prob:primal3}.

There are many dual methods that generate iterates $y^{(t)}$ converging to the optimal dual solution to~\eqref{eq:bpdn}, including the level-set method coupled with the dual conditional-gradient suboracle, as described by~\cite{aravkin2016levelset,fan2019alignment}. The resulting primal-retrieval strategy for Step~5 of \cref{alg:primal_recover} can thus be implemented by executing the following steps:
\begin{enumerate}
    \item (Top-$k$ atoms) Find the top $k$ indices $\{i_1, \dots, i_k\} \subset [n]$ of the vector $M^*y^{(t)}$ with largest absolute value and gather their corresponding signs $s_i \coloneqq \sign([M^*y^{(t)}]_i)$ for $i \in \{i_1, \dots, i_k\}$. The top-$k$ atoms are thus $\Ascr_k=\set{s_{i_1}e_{i_1},\ldots,s_{i_k}e_{i_k}}$; see~\eqref{eq:top-k}.
    \item (Retrieve coefficients) Solve the reduced problem~\eqref{eq-top-k-recover}, where in this case,
    \begin{equation*}
        c^{(t)} \in \argmin{c \in \Re_+^k} f_k(c),
        \text{where} f_k(c) = \|M[s_{i_1}e_{i_1} \cdots s_{i_k}e_{i_k}]c - b\|_2.
    \end{equation*}
    This is a nonnegative least-squares problem for which many standard algorithms are available. For example, an accelerated projected gradient descent method requires $\BigOh( m k \log(1/\epsilon) )$ iterations when $M[s_{i_1}e_{i_1} \dots s_{i_k}e_{ik}]$ has full column rank.
    \item (Termination) Step~6 of the \cref{alg:primal_recover} is implemented simply by verifying that $f_k(c^{(t)}) \leq \alpha$. (As verified by \cref{coro:polyhedral}, we may take $\epsilon=0$ in this polyhedral case.) Thus, we can terminate the algorithm and return the primal variable 
    \[x^{(t)} = [s_{i_1}e_{i_1} \dots s_{i_k}e_{ik}]c^{(t)},\]
    which is the superposition of the top-$k$ atoms.  Otherwise, the algorithm proceeds to the next iteration. 
\end{enumerate}
We describe numerical experiments for the sparse vector recovery problem in \cref{sec:bpdn}. 
\end{example}

\subsubsection{Iteration complexity}\label{sec-iteration-complexity}

In order to guarantee the quality of the recovered solution, we rely on a notion of degeneracy introduced by Nutini et al.~\cite{nutini2019active}. 

\begin{definition}
  Let $x^*$ and $y^*$, respectively, be optimal primal and dual solutions for problems \Probi\ and \Drobi, where $\Ascr$ is polyhedral. Let $\delta$ be a positive scalar. The problem pair (\Probi, \Drobi) is $\delta$-nondegenerate if for any $a\in\Ascr$, either $a\in\suppa(x^*)$ or $\ip{a}{M^*y^*} \leq \sigma\As(M^*y^*) - \delta$. 
\end{definition}

The next proposition guarantees a finite-time atom identification property when the atomic set is polyhedral.

\begin{proposition}[Finite-time atom-identification] \label{prop:polyhedral} For each problem $i
  = 1, 2, 3$, let $\{y_i^{(t)}\}_{t=1}^\infty$ be a sequence that converges to an optimal dual solution $y_i^*$. If the atomic set $\Ascr$ is polyhedral and the problem pair (\Probi, \Drobi) is $\delta$-nondegenerate, then there exists $T > 0$ such that 
  \[
   \EssCone_{\Ascr, k}(M^*y^{(t)}) \supseteq \Escr\As( M^* y_i^*) 
  \text{and}
  x^{(t)} \mbox{ is feasible for }~\eqref{eq:main_prob} 
  \quad \forall t>T.\]
  It follows that \cref{alg:primal_recover} will terminate in $T$ iterations regardless of the tolerance $\epsilon$. 
\end{proposition}

\Cref{prop:polyhedral} ensures that the atom-identification property described by~\cref{thm:p0} is guaranteed to discard superfluous atoms in a finite number of iterations as long as we have available an iterative solver that generates dual iterates converging to a solution. Thus, \cref{alg:primal_recover} is guaranteed to generate a feasible solution to~\eqref{eq:main_prob}.
The following corollary characterizes a bound on $T$ in terms of the convergence rate of the dual method.

\begin{corollary} \label{coro:polyhedral}
    For each problem $i= 1, 2, 3$, suppose the dual oracle generates iterates $y^{(t)}$ converging to optimal variable $y_i^*$ with rate
    \[d_i(y_i^{(t)}) - d_i(y_i^*) \in \Oscr\left( t^{-\alpha} \right)\]
    for some $\alpha > 0$. If the atomic set $\Ascr$ is polyhedral and the problem pair (\Probi, \Drobi) is $\delta$-nondegenerate, then \cref{alg:primal_recover} with $\epsilon=0$ terminates in $\Oscr\left(\delta^{-2/\alpha}\right)$ iterations. 
\end{corollary}

\subsubsection{Centrosymmetry and unconstrained primal recovery}

Further computational savings are possible when the atomic set $\Ascr$ is centrosymmetric, i.e.,
\begin{equation}\label{eq-centro-symmetric}
  a\in\Ascr \iff -a\in\Ascr.
\end{equation}
Centrosymmetry is a common property, and perhaps the prototypical example is the set of signed canonical unit vectors given by the set~\eqref{eq-signed-e}. Whenever centrosymmetry holds, $\cone\Ascr=\Span\Ascr$. This motivates us to replace the function $\EssCone$ with the function
\[
  \EssSpan_{\Ascr, k}(M^*y) \coloneqq \Span\Ascr_k,
\]
where $\Ascr_k$ is the collection of top-$k$ atoms defined by \cref{eq:top-k}. Thus, the primal-retrieval optimization problem~\eqref{eq-top-k-recover} reduces to the unconstrained version
\begin{equation*}
  \hat c \in \argmin{c \in \Re^k} f_k(c).
\end{equation*}

The following corollary simply asserts that the complexity results described in \cref{sec-iteration-complexity} continue to hold for centrosymmetric atomic sets when using the essential span function.
\begin{corollary}[Atom identification under centrosymmetry]
    If the atomic set $\Ascr$ is centrosymmetric and polyhedral, then \cref{prop:polyhedral} and \cref{coro:polyhedral} hold with $\EssCone$ replaced by $\EssSpan$. 
\end{corollary}

\subsection{Primal-retrieval for spectral atomic sets} \label{sec:spectral}

We formalize in this section a definition for the function $\EssCone_{\Ascr, k}$ for the case in which $\Ascr$ is a collection of rank-1 matrices, either asymmetric or symmetric, respectively:
\begin{align} \label{eq:asymmetric}
    \Ascr &= \set{uv^T \mid u \in \mR^m,\ v \in \mR^n,\ \|u\|_2=\|v\|_2 = 1}, \\ \label{eq:symmetric}
    \Ascr &= \set{vv^T \mid v \in \mR^n,\ \|v\|_2 = 1}.
\end{align}
We mainly focus on the former atomic set of asymmetric matrices because all of our theoretical results easily specialize to the symmetric case. Note that this atomic set is centrosymmetric (cf. \cref{eq-centro-symmetric}), and as we'll see below, the recovery problem is unconstrained. Later in \cref{sec:symmetric-matrices} we'll describe the recovery problem for the atomic set of symmetric matrices, which is in fact a non-centrosymmetric set.

For this section only, we work with the linear operator $M:\Re^{m\times n}\to\Re^{p\times q}$, and replace the vector of observations $b$ with the $p$-by-$q$ matrix $B$. In this context, the dual variables for one of the corresponding dual problems is a matrix of the same dimension.

Fix a feasible dual variable $Y$ and define the singular value decomposition (SVD) for its product with the adjoint of $M$ by
\begin{equation}\label{eq-svd}
   M^*(Y) = 
\begin{bmatrix}
    U_k & U_{-k}
\end{bmatrix}
\begin{bmatrix}
    \Sigma_k &          \\ 
             &  \Sigma_{-k}
\end{bmatrix}
\begin{bmatrix}
    V_k^T \\ 
    V_{-k}^T
\end{bmatrix},
\end{equation}
where $\Sigma_k$ is the diagonal matrix consisting of top-$k$ singular values of $M^*(Y)$, the matrices $U_k$ and $V_k$ contain the corresponding left and right singular vectors, and the matrices $U_{-k}$, $V_{-k}$, and $\Sigma_{-k}$ contain the remaining singular vectors and values. Then the reduced atomic set by $U_k$ and $V_k$ can be expressed as 
\[
  \Ascr_k = \set{ uv\T | u\in\range(U_k),\ v\in\range(V_k),\ \|u\|_2=\|v\|_2=1} \subset \Ascr.
\]
We construct the cone of essential atoms as the convex cone generated from the reduced atomic set $\Ascr_k$, i.e.,
\begin{equation} \label{eq:cone_spectral}
    \EssCone_{\Ascr, k}(M^*(Y)) \coloneqq \cone(\Ascr_k) = \set{U_k C V_k^T | C \in \Re^{k\times k}}.
\end{equation}
Thus, the primal-retrieval computation in Step~5 of \cref{alg:primal_recover} is then given by 
\begin{equation}\label{eq:unsymm-recovered-soln} 
  \hat X = U_k\hat C V_k^T
\end{equation}
where
\begin{equation}\label{eq-spectral-reduced}
  \hat C \in \argmin{C\in\Re^{k\times k}} f_k(C), \text{with} f_k(C) \coloneqq f\left( B - M(U_k C V_k^T) \right),
\end{equation}
is a $k\times k$ matrix obtained by solving the reduced problem~\eqref{eq:primal_recover}, which is defined over the cone generated by the essential atoms identified through the top-$k$ singular triples of $M^*(Y)$, desribed by~\eqref{eq:cone_spectral}.

\begin{example}[Low-rank matrix completion] \label{ex:mc}
    The low-rank matrix completion (LRMC) problem aims to recover a low-rank matrix from partial observations, which arises in many real applications such as recommender systems~\cite{rennie2005fast} and in a convex formulation of the phase retrieval problem~\cite{candes2013phaselift}. The LRMC problem can be expressed as 
    \begin{equation} \label{eq:lrmc}
      \find X\in\Re^{m\times n} \sut \sum_{(i,j) \in \Omega}\tfrac{1}{2}\left(X_{i,j} - B_{i,j}\right)^2 \leq \alpha \tand \rank(X) \leq k,
    \end{equation}
    where $\{B_{i,j} \mid (i,j) \in \Omega\}$ is the set of observations over the index set $\Omega$. Problem~\eqref{eq:lrmc} corresponds to~\eqref{eq:main_prob} with the objective $f=\half\|\cdot\|_2^2$, the atomic set $\Ascr$ given by~\eqref{eq:asymmetric}, and the linear operator $M$ defined by the mask 
    \[ M(X)_{i,j} = 
    \begin{cases}
        X_{i,j} &(i,j) \in \Omega, \\
        0 &\mathrm{otherwise}.
    \end{cases}
    \]
    Fazel~\cite{fazel1998approximations} popularized the convex relaxation of~\eqref{eq:lrmc} that minimizes the sum of singular values of $X$:
    \begin{equation} \label{eq:mc}
    \minimize{X\in\Re^{m\times n}}\enspace \|X\|_*  \enspace\text{subject to}\enspace \sum_{(i,j) \in \Omega}\tfrac{1}{2}\left(X_{i,j} - B_{i,j}\right)^2 \leq \alpha.
    \end{equation}
    This problem corresponds to formulation~\eqref{prob:primal3}. As with \cref{ex:bpdn}, there are many dual methods that can generate dual feasible iterates $Y^{(t)}$ converging to the dual solution of~\eqref{eq:mc}, such as a dual bundle method~\cite{fan2019bundle}. The resulting primal-retrieval strategy for Step~5 of \cref{alg:primal_recover} can be implemented by executing the following steps:
    \begin{enumerate}
        \item (Top-$k$ atoms) Compute the leading $k$ singular vectors of the matrix $M^*(Y^{(t)})$, given by $U_k^{(t)} \in \Re^{m \times k}$ and $V_k^{(t)} \in \Re^{n \times k}$ as defined by the SVD \cref{eq-svd}.
        \item (Retrieve coefficients) Solve the reduced problem~\eqref{eq-spectral-reduced}, where in this case,
        \begin{equation*}
            C^{(t)} \in \argmin{C\in\Re^{k\times k}} f_k(C),
            \text{with}
            f_k(C) \coloneqq \sum_{(i,j) \in \Omega}  \tfrac{1}{2}\left([U_k^{(t)}C(V_k^{(t)})^\intercal]_{i,j} - B_{i,j}\right)^2. 
        \end{equation*}
        This least squares problem can be solved to within $\epsilon$-accuracy in $\BigOh(( k |\Omega| + (m+n)k + k^3 )\epsilon^{-0.5})$ iterations, for example, with the FISTA algorithm~\cite{beck2009fast}. Typically, $k \ll \min\{m,n\}$, and so we expect that this reduced problem is significantly cheaper to solve than the original problem~\eqref{eq:mc}.
        \item (Termination) Step~6 of \cref{alg:primal_recover} terminates when the value of the reduced objective satisifes the condition $f_k(C^{(t)}) \leq \alpha + \epsilon$, where $\epsilon$ is some pre-defined tolerance. In that case, the algorithm returns with the primal estimate constructed from the left and right singular vectors:
        \[X^{(t)} = U_k^{(t)}C^{(t)}(V_k^{(t)})^\intercal.\]
    \end{enumerate}
    We describe numerical experiments for the low-rank matrix completion problem in \cref{sec:lrmc}. 
\end{example}

\subsubsection{Iteration complexity}

In the polyhedral case, we were able to assert through \cref{prop:polyhedral} that the optimal primal variable's atomic support could be identified in finite time. As we show here, however, finite-time identification is not possible for the spectral case. The following counterexample shows that the partial SVD of $M^*(Y)$, which we used in~\eqref{eq-svd}, is not able to give us a safe cover of the essential atoms in $\Escr\As(M^*(Y^*))$ even when this set is a singleton and $Y$ arbitrarily close to a dual solution $Y^*$. 

\begin{example}[Limitation of Partial SVD]\label{ex:partial-svd}
	Consider the problem
\begin{align} \label{eq:example_P}
  \minimize{ X \in \mR^{n \times n} }\enspace \half \| X - B \|_F^2 \enspace\text{subject to}\enspace \| X \|_* \leq 1,
\end{align} 
where
\[
  B = U \Diag( 2, 0.1, \ldots, 0.1 ) V^T \text{and} U = V = \begin{bmatrix}
    \sqrt{1 - \epsilon} &~ 0 &  ~\ldots  & -\sqrt{ \epsilon } \\
    0 & ~1  &  ~\ldots &  \\
    \vdots &  &~\ddots & \\
    \sqrt{ \epsilon } & 0 & & \sqrt{1 - \epsilon}
    \end{bmatrix}_{n \times n}
\]
for some fixed $\epsilon \in (0,1)$. The dual problem is
\begin{align}
  \minimize{Y \in \mR^{n \times n} }\enspace \half \| Y - B \|_F^2 - \half \| B \|_F^2 + \|Y\|_2. \label{eq:example_D}
\end{align}
The solutions for the dual pair~\eqref{eq:example_P} and~\eqref{eq:example_D} are
\[
  X^* =  U \Diag( 1, 0, \ldots, 0 ) V^T \text{and} Y^* = B - X^* = U \Diag( 1, 0.1, \ldots, 0.1 ) V^T.
\]

In this pair of problems, the linear operator $M$ is simply the identity map, and the cone of essential atoms described by~\eqref{eq:cone_spectral} depends only on the dual variable $Y$.  Let $u_1$ and $v_1$, respectively, be the first columns of $U$ and $V$. Evidently, the support of $X^*$ coincides with the essential atoms of $Y^*$, and moreover, the support is a unique singleton. In other words,
\[
  \suppa(X^*) = \Escr\As(Y^*) = \{u_1 v_1^T\}.
\]
We construct the following dual feasible solution
\[
  \widehat{Y} = \Diag(1, 0.1, \ldots, 0.1).
\]
Because $\widehat Y$ is diagonal, its left and right singular vectors $\widehat{U}$ and $\widehat{V}$ are given by $\widehat{U} = \widehat{V} = I = [e_1, e_2, \ldots, e_n]$. Note also that the top singular vector $u_1 = [\sqrt{1 - \epsilon}, 0, \ldots, \sqrt{\epsilon}]^T$ lies in the span of the basis vectors $e_1$ and $e_n$ that constitute the top and bottom singular vectors of $\widehat{Y}$. Therefore, any top-$r$ SVD of $\widehat{Y}$, with $r < n$, cannot be used to recover exactly a primal solution $X^*$. Moreover, $\| \widehat{Y} - Y^* \|_F = \BigOh( \sqrt{\epsilon} )$ for any $\epsilon\in(0,1)$, which in effect implies that it's impossible to recover exactly the true solution even with an arbitrarily accurate dual approximation $\widehat{Y}$.
\end{example}


This last example motivates our study of the quality with which a partial SVD of a given feasible dual solution $M^*(Y)$ can be used to approximate the support $\Escr\As(M^*(Y^*))$. The next result measures the difference between $\suppa(x^*)$ and $\EssCone_{\Ascr, k}(M^*(Y))$ using one-sided Hausdorff distance.

\begin{proposition}[Error in truncated SVD] \label{thm:svd_approx_score}
  Let $Y$ be feasible for one of the dual problems~\Drobi\ for $i=1,2,3$.
Let $\{\sigma_j\}_{j=1}^{\min\{n,m\}}$ be the singular values satisfying 
$\sigma_1 \geq \dots \geq \sigma_{\min\{n,m\}}$, where we assume $\sigma_1 > \sigma_{k+1}$. Let $Z:=M^*(Y)$ and let $\EssCone_{\Ascr, k}(Z)$ denote the cone generated according to equation~\eqref{eq:cone_spectral}. Then 
  \[
    \dist( \suppa(X^*),\  \EssCone_{\Ascr, k}(Z)) \leq
    \dist( \Escr\As( Z, \epsilon_i ),\ \EssCone_{\Ascr, k}(Z)  ) \leq \sqrt{ 2 \min\left\{ \frac{\epsilon_i}{ \sigma_1 - \sigma_{k+1} }, 1 \right\} },
  \]
  where each $\epsilon_i$ is defined in \cref{thm:p0} and the function 
  \[  
  \dist( \Ascr_1, \Ascr_2 ) \coloneqq \adjustlimits\sup_{a_1 \in \Ascr_1} \inf_{a_2 \in \Ascr_2 } \| a_1 - a_2 \|_F.
    \]
  is the one-sided Hausdorff distance between sets $\Ascr_1$ and $\Ascr_2$.
\end{proposition}


Oustry~\cite[Theorem~2.11]{Oustry00} developed a related result based on the two-sided Hausdorff distance. Directly applying Oustry's result to our context results in a bound on the order $ \mathcal{O}( \sqrt{ \epsilon/( \sigma_k - \sigma_{k+1} ) } ) $, which is looser than the bound shown in \cref{thm:svd_approx_score} because $\sigma_1 \geq \sigma_{k} \geq \sigma_{k+1}$. 



Finally, we show the error bound for primary recovery. 
\begin{proposition} \label{prop:spectral}
    Assume that $f \geq 0$ and $f(0) = 0$. Let $X^*$ and $Y$, respectively, be primal optimal and dual feasible for one of the primal-dual pairs \Probi\ and \Drobi, for $i=1,2,3$.
Let $\{\sigma_j\}_{j=1}^{\min\{n,m\}}$ be the singular values satisfying 
$\sigma_1 \geq \dots \geq \sigma_{\min\{n,m\}}$, where we assume $\sigma_1 > \sigma_{k+1}$.  Let $\EssCone_{\Ascr, k}(M^*(Y))$ denote the cone generated according to equation~\eqref{eq:cone_spectral}. Let $\hat X$ be the solution recovered via~\eqref{eq:unsymm-recovered-soln}.
Then
\[
    f( B - M (\hat X )) \le f( B - M (X^*) ) + \BigOh\left( \sqrt{ \frac{\epsilon_i}{ \sigma_1 - \sigma_{k+1} } } \right),
\]
where each $\epsilon_i$ is defined in \cref{thm:p0}.
\end{proposition}

\Cref{prop:spectral} characterizes the error bound for our primal-retrieval strategy when $\Ascr$ is spectral. Our next corollary shows that \cref{alg:primal_recover} can terminate in polynomial time with any tolerance $\epsilon > 0$.  

\begin{corollary} \label{coro:spectral}
    For one of the problems \Drobi, $i=1,2,3$, suppose that a dual oracle generates iterates $Y^{(t)}$ converging to optimal variable $Y^*$ with convergence rate
    \[d_i(Y^{(t)}) - d_i(Y^*) \in \Oscr\left( t^{-\alpha} \right)\]
    for some $\alpha > 0$. If the atomic set $\Ascr$ is spectral then \cref{alg:primal_recover} with $\epsilon>0$ will terminate in $\Oscr\left(\epsilon^{-4/\alpha}\right)$ iterations. 
\end{corollary}

\subsubsection{Non-centrosymmetry and constrained primal recovery} \label{sec:symmetric-matrices}

We now consider the case in which the atomic set $\Ascr$ given by \eqref{eq:symmetric}, which is not centrosymmetric. As we show below, the corresponding primal recovery problem is constrained.

Fix a feasible dual variable $Y$ and define the eigenvalue decomposition for its product with the adjoint of $M$ by
\begin{equation*}
   M^*(Y) = 
\begin{bmatrix}
    V_k & V_{-k}
\end{bmatrix}
\begin{bmatrix}
    \Sigma_k &          \\ 
             &  \Sigma_{-k}
\end{bmatrix}
\begin{bmatrix}
    V_k\T \\ 
    V_{-k}\T
\end{bmatrix},
\end{equation*}
where $V_k$ and the diagonal matrix $\Sigma_k$, respectively, contain the top-$k$ eigenvectors and eigenvalues of $M^*(Y)$, and $V_{-k}$ and the diagonal matrix $\Sigma_{-k}$, respectively, contain the remaining eigenvectors and eigenvalues. Then the reduced atomic set by $V_k$ can be expressed as 
\[
  \Ascr_k = \set{ vv\T | v\in\range(V_k),\ \|v\|_2=1} \subset \Ascr.
\]
The convex cone of essential atoms generated from the reduced atomic set $\Ascr_k$ is given by
\begin{equation*}
    \EssCone_{\Ascr, k}(M^*(Y)) \coloneqq \cone(\Ascr_k) = \set{V_k C V_k\T | C \in \Re^{k\times k},\ C \succeq 0}.
\end{equation*}
The recovery problem \eqref{eq-spectral-reduced} then becomes constrained, i.e., 
\[
  \hat C \in \argmin{C\in\Re^{k\times k},\ C \succeq 0} f_k(C).
\]

\section{Numerical experiments}

We conduct several numerical experiments on both synthetic and real-world datasets to empirically verify the effectiveness of our proposed primal-retrieval strategy. In \cref{sec:bpdn}, we describe experiments on the basis pursuit denoising problem (\cref{ex:bpdn}), which shows the performance of our strategy on polyhedral atomic set. In \cref{sec:lrmc}, we apply our primal-retrieval technique to the low-rank matrix completion problem (\cref{ex:mc}) and test the effectiveness of our proposed method on the spectral atomic set. In \cref{sec:rpca}, we describe experiments on 
a image preprocessing problem, where the atomic set $\Ascr$ is the sum of a polyhedral atomic set and a spectral atomic set. This shows that our strategy can be applied to more complicated cases. For all experiments, we implement the level-set method proposed by Aravkin et al.~\cite{aravkin2016levelset} where we only store the dual variable $y$. We implement the level-set method and our primal-retrieval strategy in the Julia language~\cite{bezanson2017julia} and our code is publicly available at \url{https://github.com/MPF-Optimization-Laboratory/AtomicOpt.jl}. All the experiments are carried out on a Linux server with 8 CPUs and 64 GB memory. 

\subsection{Basis pursuit denoise} \label{sec:bpdn}

The experiments in this section include a selection of five relevant basis pursuit problems from the Sparco collection~\cite{BergFrieHennHerrSaabYilm:2008} of test problems. The chosen problems are all real-valued and suited to one-norm regularization. Each problem in the collection includes a linear operator $M:\Re^n \to \Re^m$ and a right-hand-side vector $b \in \Re^m$. \Cref{tab:sparco_info} summarizes the selected problems. We compare the results with SPGL1~\cite{berg2008probing}. In all problems, we set $\alpha = 10^{-3}\cdot\|b\|$. The results are shown in \Cref{tab:bpdn} where \texttt{nMat} denotes the total number of matrix-vector products with $M$ or $M^*$. 
As we can observe from \Cref{tab:bpdn}, the level-set algorithm equipped with our primal-retrieval technique can obtain an $\epsilon$-feasible solution within a small number of iterations, which is consistent with the finite-time identification property described by \Cref{prop:polyhedral}. We also observe that level-set method coupled with the primal-retrieval strategy can converge faster than SPGL1 with its default stopping criterion. This suggests that our primal-retrieval technique is both a memory-efficient method for obtaining approximal primal solutions with provable error bounds, and is also a practical technique that allow the optimization algorithm to stop early.

\begin{table}[t]
    \centering
    \begin{tabular}{llllll}
        \toprule
        Problem             &   ID &    $m$ &   $n$ &   $\|b\|$ &   $M$\\
        \midrule
        \texttt{blocksig}   &   2  &    1024&   1024&   7.9e+1  &   wavelet\\
        \texttt{cosspike}   &   3  &    1024&   2048&   1.0e+2  &   DCT\\
        \texttt{gcosspike}  &   5  &    300 &   2048&   8.1e+1  &   Gaussian ensemble + DCT\\
        \texttt{sgnspike}   &   7  &    600 &   2560&   2.2e+0  &   Gaussian ensemble\\
        \texttt{spiketrn}   &   903&    1024&   1024&   5.7e+1  &   1D convolution\\
        \bottomrule
    \end{tabular}
    \captionsetup{justification=centering}
    \caption{The Sparco test problems used.}
    \label{tab:sparco_info}
\end{table}

\begin{table}[t]
    \centering
    \begin{tabular}{lrrr}
        \toprule
        Problem             &   \texttt{nnz}(x) &    \texttt{nMat}(SPGL1) & \texttt{nMat}(level-set+PR)\\
        \midrule
        \texttt{blocksig}   &   71              &     22                  &   5\\
        \texttt{cosspike}   &   113             &     77                  &   71\\
        \texttt{gcosspike}  &   113             &    434                  &   141\\
        \texttt{sgnspike}   &   20              &    44                   &   21\\
        \texttt{spiketrn}   &   35              &    4761                 &   1888\\
        \bottomrule
    \end{tabular}
    \captionsetup{justification=centering}
    \caption{Basis pursuit denoise comparisons.}
    \label{tab:bpdn}
\end{table}

\subsection{Low-rank matrix completion} \label{sec:lrmc}

For this atomic set, we conduct an experiment similar to that carried out by Cand\'es and Plan~\cite{candes2010matrix}. We retrieved from the website of National Centers for Environmental Information\footnote{\url{https://www.ncei.noaa.gov}} a 6798-by-366 matrix $X$ whose entries are daily average temperatures at 6798 different weather stations throughout the world in year 2020. The temperature matrix $X$ is approximately low rank in the sense that $\|X - X_5\|_F / \|X\|_F \approx 24\%$, where $X_5$ is the matrix created by truncating the SVD after the top 5 singular values. 

To test the performance of our matrix completion algorithm, we subsampled 50\% of $X$ and then recovered an estimate $\hat X$.  The solution gives a relative error of $\|X - \hat X\|_F / \|X\|_F \approx 30\%$. The result is shown in \cref{fig:matrix_completion}.  As we can see from \Cref{fig:matrix_completion}, the recovery error exhibits a positive correlation with the duality gap, both the duality gap and the recovery gap decrease as the number of iteration increase. The observation in this experiment is consistent with our theory (\Cref{prop:spectral}).


\subsection{Robust principal component analysis} \label{sec:rpca}

In this section, we show that our primal-retrieval strategy can be applied to more complicated atomic sets besides polyhedral and spectral. We conduct the the similar experiment as in \cite{candes2011robust}. Face recognition algorithms are sensitive to shadows on faces, and therefore it's necessary to remove illumination variations and shadows on the face images. We obtained face images from the Yale B face database~\cite{georghiades2001few}. We show the original faces in \cref{fig:face}, where each face image was of size $192\times 168$ with 64 different lighting conditions. The images were then reshaped into a matrix $M \in \Re^{32256 \times 64}$. Because of the similarity between faces and the sparse structure of the shadow, the matrix $M$ can be approximately decomposed into two components, i.e., 
\[M \approx L + S,\]
where $L$ is a low-rank matrix corresponding to the clean faces and $S$ is sparse matrix corresponding to the shadows. Based on the work by Fan et al.~\cite{fan2020polar}, we know that such decomposition can be obtained via solving the following convex optimization problem:
\begin{equation} \label{eq:rpca}
    \min_{L, S} \enspace \max\{\|L\|_*,\ \lambda \|S\|_1\}  \enspace \st  \enspace  \|L + S - M\| \leq \alpha.
\end{equation}
By \cite[Proposition~7.3]{fan2019alignment}, we know that \eqref{eq:rpca} is equivalent to 
\begin{equation*}
    \min_{X} \enspace \gauge\As(X)  \enspace \st  \enspace  \|X - M\| \leq \alpha,
\end{equation*}
with $X=L+S$ and where $\Ascr = \lambda \Ascr_1 + \Ascr_2$, $\Ascr_1 = \set{uv\T \mid u \in \mR^m,\ v \in \mR^n,\ \|u\|_2=\|v\|_2 = 1}$ and $\Ascr_2 = \set{\pm e_ie_j\T \mid i \in [m], j \in [n]}$. The recovered low-rank component is shown in \cref{fig:face_shadow}. As we can see from the figure, most of the shadow are successfully removed. This experiment suggests that our primal-retrieval technique can potentially be used for more complex atomic set and allow the underlying the dual-algorithm to produce satisfactory result within a reasonable number of iterations.

\begin{figure}[t] \label{fig:rpca}
  \begin{subfigure}{.35\textwidth}
    \centering
    \includegraphics[width=\linewidth]{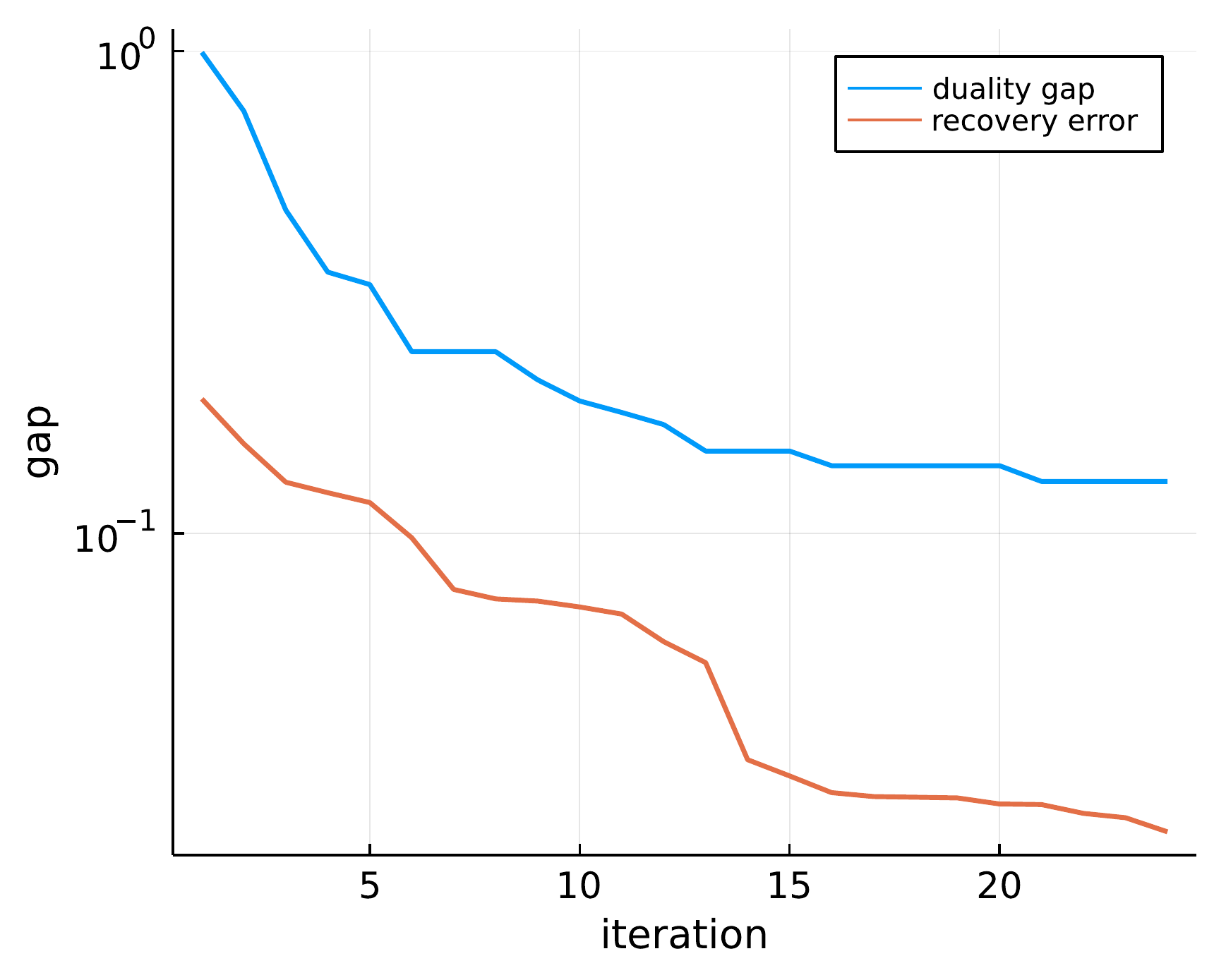}
    \captionsetup{justification=centering}
    \caption{Recovery error curve.}
    \label{fig:matrix_completion}
  \end{subfigure}
  \hfill
  \begin{subfigure}{.25\textwidth}
    \centering
    \includegraphics[width=\linewidth]{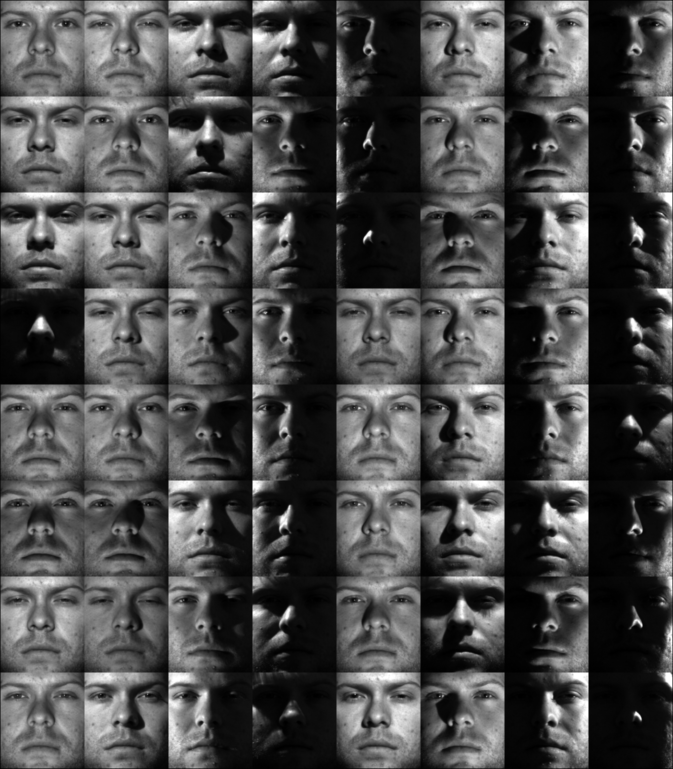}
    \captionsetup{justification=centering}
    \caption{Original faces.}
    \label{fig:face}
  \end{subfigure}
  \hfill
  \begin{subfigure}{.25\textwidth}
    \centering
    \includegraphics[width=\linewidth]{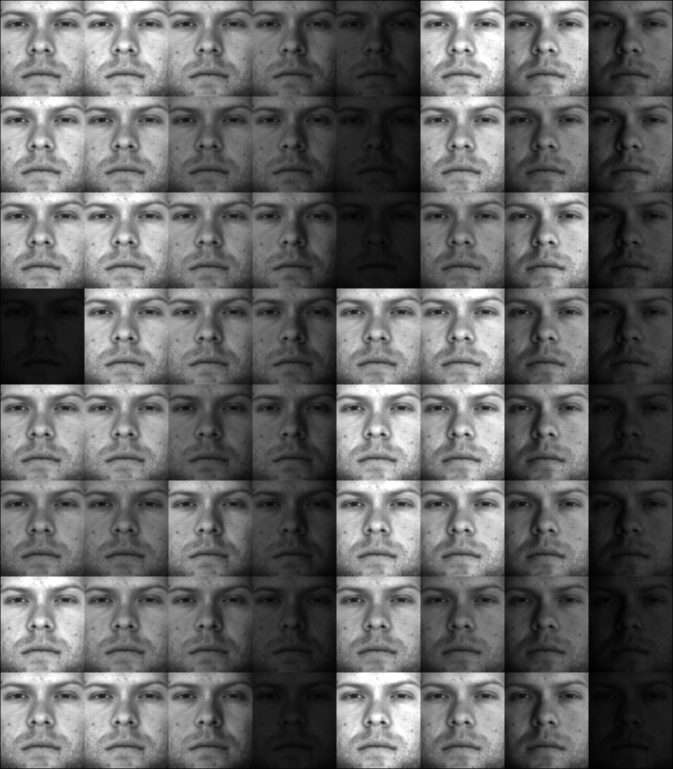}
    \captionsetup{justification=centering}
    \caption{Shadow-removed faces.}
    \label{fig:face_shadow}
  \end{subfigure}
  \caption{The left figure (\ref{fig:matrix_completion}) shows the result of the matrix completion experiment. The middle and right figures (\ref{fig:face}, \ref{fig:face_shadow}) are for the robust principal component analysis experiment.}
  \end{figure}

\section{Conclusion}
In this work, we proposed a simple primal-retrieval strategy for atomic-sparse optimization. We demonstrate both theoretically and empirically that our proposed strategy can obtain good solutions to the cardinality-constrained problem given a dual-based algorithm converging to the optimum dual solution. 

Further research opportunities remain, particularly for designing meaningful primal-retrieval strategies for non-polyhedral and non-spectral atomic sets. The primal-retrieval technique developed in this work is algorithm-agnostic, and it is an open questionn if it's possible to develop more efficient primal-retrieval approaches tailored to specific optimization algorithms, such as the conditional-gradient method.



\bibliographystyle{plain+eid}
\bibliography{shorttitles, master, friedlander}

\appendix

\section*{Appendix}
\section{Derivation of duals}\label{app:duals}

We derive the dual problems~\eqref{prob:dual1},~\eqref{prob:dual2} and~\eqref{prob:dual3} using the Fenchel\textendash Rockafellar duality framework. We use the following result.
\begin{theorem}[\protect{\cite[Corollary~31.2.1]{rockafellar1970convex}}]\label{thm-fenchel}
  Let $f_1:\Re^n\to\Re$ and $f_2:\Re^m\to\Re$ be two closed proper convex functions and let $M$ be a linear operator from $\Re^n$ to $\Re^m$, then
  \[\inf_{x \in \Re^n} f_1(x) + f_2(Mx) = \inf_{y \in \Re^m} f_1^*(M^*y) + f_2^*(-y).\]
  If there exist $x$ in the interior of $\dom f_1$ such that $Mx$ in the interior of $\dom f_2$, then strong duality holds, namely both infima are attained. 
\end{theorem}
We also need a result that describes the relationship between gauge, support, and indicator functions. 
\begin{proposition}[\protect{\cite[Proposition~3.2]{fan2019alignment}}] \label{prop-conjugate-gauge}
  Let $C\subset\Re^n$ be a closed convex set that contains the origin. Then 
  \[\gauge_C = \sigma_{C^\circ}=\delta_{C^\circ}^*.\]
\end{proposition}

For problem~\eqref{prob:primal1}, let
\[
  f_1 := \lambda\gauge\As \enspace\text{and}\enspace f_2 := f(b - \cdot)
\]
By the properties of conjugate functions and~\cref{prop-conjugate-gauge}, we obtain 
\[
  f_1^* =  \delta_{(\frac{1}{\lambda}\Ascr)^\circ} = \delta_{\{x\mid \sigma\As(x)\leq\lambda\}} \enspace\text{and}\enspace f_2^* = \ip{b}{\cdot} + f^*(-\cdot).
\]
Then by~\cref{thm-fenchel}, we can get the dual problem for~\eqref{prob:primal1} as
\[\minimize{y\in\Re^m} f^*(y) - \ip{b}{y} \enspace\text{subject to}\enspace \sigma\As(M^*y)\leq\lambda.\]

For~\eqref{prob:primal2},
\[f_1 = \delta_{\gauge\As\leq\tau} = \delta_{\tau\Ascr} \enspace\text{and}\enspace f_2 = f(b - \cdot).
\]
By the properties of conjugate functions and~\cref{prop-conjugate-gauge}, we obtain 
\[f_1^* = \sigma_{\tau\Ascr} = \tau\sigma\As \enspace\text{and}\enspace f_2^* = \ip{b}{\cdot} + f^*(-\cdot).
\]
Then by~\cref{thm-fenchel}, it follows that the dual problem for~\eqref{prob:primal2} is 
\[\minimize{y\in\Re^m} f^*(y) - \ip{b}{y} + \tau\sigma\As(M^*y).\]

For~\eqref{prob:primal3},
\[
  f_1 = \gauge\As \enspace\text{and}\enspace f_2 = \delta_{\{x\mid f(b - x)\leq\alpha\}}.
\]
By the properties of conjugate functions and~\cref{prop-conjugate-gauge}, we can get that 
\[
  f_1^* = \delta_{\{x\mid \sigma\As(x)\leq1\}} \enspace\text{and}\enspace f_2^* = \sigma_{\{f(b - x)\leq\alpha\}}.\]
Then by~\cite[Example~E.2.5.3]{hiriart-urruty01}, we know that the support function of the sublevel set is 
\[f_2^* = \sigma_{\{x\mid f(b - x)\leq\alpha\}} = \min_{\beta > 0} \beta\left(f^*\left(-\frac{\cdot}{\beta}\right) + \alpha\right) + \ip{b}{\cdot}.\]
Finally, by~\cref{thm-fenchel}, we can get the dual problem for~\eqref{prob:primal3} as
\[\minimize{y\in\Re^m,\ \beta > 0} \beta \left( f^*\left( \frac{y}{\beta} \right) + \alpha \right) - \ip{b}{y} \enspace \text{subject to}\enspace \sigma\As(M^* y) \leq 1.\]

\section{Proof of Theorem~\ref{thm:p0}}\label{app:main_proof}

The proof of this Theorem relies on the following duality property between smoothness and strong convexity.
\begin{lemma}[\protect{\cite[Theorem~6]{kakade2009duality}}] \label{lemma:conjugate}
\label{lemma:smooth_cvx_conjugate}
   If $f$ is $L$-smooth, then $f^*$ is $\frac{1}{L}$-strongly convex.
\end{lemma}
\begin{proof}[Proof of Theorem~\ref{thm:p0}]
\begin{itemize}
  \item[a)] Let $y^*$ denote the optimal dual variable for~\ref{prob:dual1}. First, we show that $\|y - y^*\|$ can be bounded by the duality gap. Let $g(y) \coloneqq f^*(y) - \ip{b}{y}$. By~\cref{lemma:conjugate}, $f^*$ is $\frac{1}{L}$-strongly convex, and it follows that $g$ is also $\frac{1}{L}$-strongly convex. By the definition of strong convexity, 
  \[\forall s \in \partial g(y^*), \enspace g(y) \geq g(y^*) + \ip{s}{y- y^*} + \frac{1}{2L}\|y - y^*\|^2.
  \]
  Optimality requires that 
  \[\exists s \in \partial g(y^*), \enspace \ip{s}{y- y^*} \geq 0 \quad \forall y \enspace\text{s.t.}\enspace \sigma\As(M^*y)\leq\lambda.\]
  Therefore, reordering the inequality gives
  \begin{align*}
     \|y - y^*\| & \leq \sqrt{2L(g(y) - g(y^*))}
     \quad \forall x \in \Re^n.
  \end{align*} 

  Next, we show that $\Escr\As(M^*y^*) \subseteq \Escr\As(M^*y, \epsilon_1)$. For any $a \in \Escr\As(M^*y^*)$, 
  \begin{align*}
    \ip{a}{M^*y} &= \sigma\As(M^*y^*) + \ip{Ma}{y - y^*}
    \\&\geq \sigma\As(M^*y^*) - \left(\max\limits_{a \in \Ascr}\|Ma\|\right)\|y - y^*\|
    \\&\geq \sigma\As(M^*y^*) - \left(\max\limits_{a \in \Ascr}\|Ma\|\right)\sqrt{2L(g(y) - g(y^*))}
    \\&\geq \sigma\As(M^*y) - \epsilon_1,
  \end{align*}
  where the last inequality follows from the definition of $\epsilon_1$ in \Cref{thm:p0} and the fact that $\sigma\As(M^*y^*) = \lambda$ and $y$ is feasible for \eqref{prob:dual1}. 

  \item[b)] Let $y^*$ denote the optimal dual variable for~\ref{prob:dual2}. First, we show that $\|y - y^*\|$ can be bounded by the duality gap. Let $g(y) \coloneqq f^*(y) - \ip{b}{y} + \tau\sigma\As(M^*y)$. By~\cref{lemma:conjugate}, $f^*$ is $\frac{1}{L}$-strongly convex, and it follows that $g$ is also $\frac{1}{L}$-strongly convex. 
  By the definition of strongly convex, 
  \[\forall s \in \partial g(y^*), \enspace g(y) \geq g(y^*) + \ip{s}{y- y^*} + \frac{1}{2L}\|y - y^*\|^2.\]
  By optimality, $0 \in \partial g(y^*)$. Reordering the inequality to deduce that
  \[\|y - y^*\|_2 \leq \sqrt{2L(g(y) - g(y^*))}.\]

  Next, we show that $\Escr\As(M^*y^*) \subseteq \Escr\As(M^*y, \epsilon_2)$. For any $a \in \Escr\As(M^*y^*)$, 
  \begin{align*}
    \ip{a}{M^*y} &\geq \sigma\As(M^*y^*) - \left(\max\limits_{a \in \Ascr}\|Ma\|\right)\|y - y^*\|
    \\&= \sigma\As(M^*y) - \left(\sigma\As(M^*y) - \sigma\As(M^*y^*)\right) - \left(\max\limits_{a \in \Ascr}\|Ma\|\right)\|y - y^*\|
    \\&\geq \sigma\As(M^*y) - 2\left(\max\limits_{a \in \Ascr}\|Ma\|\right)\|y - y^*\|
    \\&\geq \sigma\As(M^*y) - \epsilon_2,
  \end{align*}
  where the last inequality follows from the definition of $\epsilon_2$ in \Cref{thm:p0}.

  \item[c)] Let $(y^*, \beta^*)$ denote the optimal dual variables for~\ref{prob:dual3}. First, we show that $\|y - y^*\|$ can be bounded by the duality gap. Let 
  \[g(y) \coloneqq \beta^*f^* \left(\frac{y}{\beta^*} \right) + \beta^*\alpha - \ip{b}{y}.\] 
  By~\cref{lemma:conjugate}, $f^*$ is $\frac{1}{L}$-strongly convex, and it's not hard to check that $g$ is $\frac{1}{\beta^*L}$-strongly convex. By the definition of strongly convex,
  \[\forall s \in \partial g(y^*), \enspace g(y) \geq g(y^*) + \ip{s}{y- y^*} + \frac{1}{2\beta^*L}\|y - y^*\|^2.\]
  By optimality,
  \[\exists s \in \partial g(y^*), \enspace\ip{s}{y- y^*} \geq 0 \quad \forall y \enspace\text{s.t.}\enspace \sigma\As(M^*y)\leq1.\]
  Reorder the inequality to deduce that 
  \[\|y - y^*\| \leq \sqrt{2\beta^*L(g(y) - g(y^*))}.\]

  Since $\beta^*$ is unknown to us, we will then get an upper bound for $d_3(y, \beta^*)$. Fix $y$, let $h(\beta) = d_3(y, \beta)$. By the property of perspective function, we know that $h$ is convex. Then it follows that 
  \[d_3(y, \beta^*) \leq \max\set{d_3(y, \underline{\beta}), d_3(y, \overline{\beta})}.\]
  Therefore,
  \[ \|y - y^*\| \leq \sqrt{2\overline{\beta}L \left( \max\set{d_3(y, \underline{\beta}), d_3(y, \overline{\beta})} - d_3(y^*, \beta^*)\right) } .\]

  Finally, we show that $\Escr\As(M^*y^*) \subseteq \Escr\As(M^*y, \epsilon_3)$. For any $a \in \Escr\As(M^*y^*)$, 
  \begin{align*}
    \ip{a}{M^*y} &\geq \sigma\As(M^*y^*) - \left(\max\limits_{a \in \Ascr}\|Ma\|\right)\|y - y^*\|
    \\&= \sigma\As(M^*y) - \left(\sigma\As(M^*y) - \sigma\As(M^*y^*)\right) - \left(\max\limits_{a \in \Ascr}\|Ma\|\right)\|y - y^*\|
    \\&\geq \sigma\As(M^*y) - 2\left(\max\limits_{a \in \Ascr}\|Ma\|\right)\|y - y^*\|
    \\&\geq \sigma\As(M^*y) - \epsilon_3.
  \end{align*}
\end{itemize}
\end{proof}

\section{Upper and lower bound for $\beta^*$}
\label{app:bounds}
First, we consider~\eqref{prob:dual3}. Let $w = y/\beta$, then~\eqref{prob:dual3} can be equivalently expressed as 
\[\minimize{w} \inf_{\beta>0} \beta(f^*(w) - \ip{b}{w} + \alpha) \enspace\text{subject to}\enspace \sigma\As(M^* w) \leq \beta.\]
Fix $\beta = \beta^*$, then~\eqref{prob:dual3} can be expressed as 
\begin{equation} \label{prob:dual3equiv}
  \minimize{w} f^*(w) - \ip{b}{w} \enspace\text{subject to}\enspace \sigma\As(M^* w) \leq \beta^*.
\end{equation}
Now compare~\eqref{prob:dual3equiv} with~\eqref{prob:dual1} to conclude that they are equivalent when $\lambda = \beta^*$. It thus follows that~\eqref{prob:primal3} is equivalent to  
\begin{equation*}
    \minimize{x} f(b - Mx) + \beta^*\gauge\As(x). 
\end{equation*}

Next, consider using the level-set method~\cite{aravkin2016levelset} with bisection to solve~\eqref{prob:primal3}. There exists $\tau^* > 0$ such that~\eqref{prob:primal3} is equivalent to
\begin{equation*}
    \minimize{x} f(b - Mx) \enspace\text{subject to}\enspace \gauge\As(x)\leq\tau^*. 
\end{equation*}
With the level-set method, we are able to get $(x_1, \tau_1)$ and $(x_2, \tau_2)$ such that $\tau_1 \leq \tau^* \leq \tau_2$ and $x_i$ is the optimum for 
\begin{equation*}
    \minimize{x} f(b - Mx) \enspace\text{subject to}\enspace \gauge\As(x)\leq\tau_i, 
\end{equation*}
for $i = 1, 2$. Then there exits $\beta_1$ and $\beta_2$ such that $\beta_1 \geq \beta^* \geq \beta_2$ and $x_i$ is optimal for 
\begin{equation*}
    \minimize{x} f(b - Mx) + \beta_i\gauge\As(x),
\end{equation*}
for $i = 1, 2$.

Finally, by~\cite[Theorem~5.1]{fan2019alignment} we can conclude that 
\[\beta_i = \sigma\As(M^*\nabla f(b - Mx_i)) \text{for} i = 1,2.\]
Therefore, we can get upper and lower bounds for $\beta^*$ via level-set method with bisection. Moreover, by strong duality and convergence of the bisection method, the gap between $\beta_1$ and $\beta_2$ will converge to zero.

\section{Proof of Proposition~\ref{prop:polyhedral}}
\label{app:prop_proof}
\begin{proof}
  First, we show that for any $y_i$ such that $\| y_i - y_i^* \| \leq \frac{\delta}{4\| M \|_\Ascr}$, the condition 
  \[\Face\As( M^* y_i^*) \subseteq \EssCone(\Ascr, M, y_i, k)\]
  holds. By \cref{ass:blanket} and the definition of $\delta$-nondegeneracy, we know that 
  \begin{align}
      |  \Face\As( M^* y_i^*) | = k, \quad \text{and} \quad \langle Ma, y_i^* \rangle \leq \sigma_{\Ascr}( M^* y_i^* ) - \delta \quad \forall a \notin \Face\As( M^* y_i^*). \label{eq:Aprime}
  \end{align}
  For any $a \in \Face\As( M^* y_i^*)$, we have
  \begin{align*}
      & \langle a, M^* y_i \rangle \\
      \geq~& \langle a, M^* y_i^* \rangle - | \langle M a, y_i^* - y_i \rangle | \\
      \geq~& \sigma_{\Ascr}( M^* y_i^* ) - \| M \|_\Ascr \frac{\delta}{4\|M\|_{\Ascr}} \quad \text{(by the condition $\| y_i - y_i^* \| \leq \frac{\delta}{4\| M \|_\Ascr}$)} \\
      \geq~& \sigma_{\Ascr}( M^* y_i^* ) - \frac{\delta}{4}.
  \end{align*}
  On the other hand, for any $a' \notin \Face\As( M^* y_i^*)$, we have 
  \begin{align*}
      & \langle a', M^* y_i \rangle \\
      \leq~ & \langle a', M^* y_i^* \rangle + | \langle M a', y_i^* - y_i \rangle | \\
      \leq~ & \langle a', M^* y_i^* \rangle + \frac{\delta}{4} \\
      \leq~ & \sigma_{\Ascr}( M^* y_i^* ) - \delta + \frac{\delta}{4} \quad \text{(By \cref{eq:Aprime})} \\
      =~& \sigma_{\Ascr}( M^* y_i^* ) - \frac{3\delta}{4}.
  \end{align*}
    Therefore, 
    \begin{align*}
        \langle a, M^* y_i \rangle > \langle a', M^* y_i \rangle  \quad \forall a \in \Face\As( M^* y_i^*) \text{and} a' \notin \Face\As( M^* y_i^*).
    \end{align*}
    Notice that $\EssCone(\Ascr, M, y_i, k)$ contains only the atoms that corresponds to the $k$ largest $\langle a, M^* y_i \rangle$. Therefore $\Face\As( M^* y_i^*) \subseteq \EssCone(\Ascr, M, y_i, k)$.
    
    By the assumption $y_i^{(t)} \to y_i^*$. For $i \in \{1,2,3\}$, we know there exist $T_i > 0$ such that $\| y_i^{(t)} - y_i^* \| < \frac{\delta}{4\| M \|_\Ascr}$ for all $t > T_i$. Therefore $\Face\As( M^* y_i^*) \subseteq \EssCone(\Ascr, M, y_i^{(t)}, k)~~\forall t > T_i$ and we complete the proof.
\end{proof}

\section{Proof for \Cref{thm:svd_approx_score}}
\begin{proof}
  First, we derive a monotonicity property of $\dist(\cdot, \cdot)$. By the definition of $\dist(\cdot, \cdot)$, it follows that
  \begin{equation}
    \dist(A, C) \leq \dist(B, C) \quad \forall A, B, C \subseteq \mR^{n \times m} ~\text{such that}~ A \subseteq B. \label{eq:distIneq}
  \end{equation}
  For any $i \in \{1,2,3\}$, we know that $ \Sscr_{\Ascr}(X^*) \subseteq  \Escr\As( Z, \epsilon_i )$ by \Cref{thm:p0}. Then by \cref{eq:distIneq}, we have 
  \[
    \dist( \Sscr_{\Ascr}(X^*),\ \EssCone_{\Ascr, k}(Z)  ) \leq \dist( \Escr\As( Z, \epsilon_i ), \ \EssCone_{\Ascr, k}(Z)  ).
  \]
  For any $\Ascr_1, \Ascr_2 \subseteq \Ascr$,
  \begin{align*}
    \rho( \Ascr_1, \Ascr_2 ) = \sqrt{ \adjustlimits\sup_{a_1 \in \Ascr_1} \inf_{a_2 \in \Ascr_2} \| a_1 - a_2 \|_F^2 } = \sqrt{ 2 - 2 \left( \adjustlimits \inf_{a_1 \in \Ascr_1} \sup_{a_2 \in \Ascr_2} \ip{a_1}{a_2} \right) },
  \end{align*}
  where the second equality holds since $\| a_1 \|_F = \| a_2 \|_F = 1$ by the definition of $\Ascr$. Define $\Ascr_1 = \Escr\As( Z, \epsilon_i )$ and $\Ascr_2 = \EssCone_{\Ascr, k}(Z) = \Set{ U_r p q^T V_r^T | \|p\|_2 = \|q\|_2 = 1 }$, where $U_r, V_r$ are the top-$r$ singular vectors of $M^*y$. Let $k \coloneqq \min\{n,m\}$, $\Cscr_1 = \{(p, q) \mid \sum_{i=1}^{k} \sigma_i p_i q_i \geq \sigma_1 - \epsilon_i,~ \|p\|_2 = \|q\|_2 = 1, ~p,q \in \mR^{ k }\}$ and $\Cscr_2 = \{(\hat{p}, \hat{q}) \mid \|\hat{p}\|_2 = \|\hat{q}\|_2 = 1, ~\hat{p},\hat{q} \in \mR^{ r }\}$, then
  \begin{align}
    \rho( \Ascr_1, \Ascr_2  ) & = \sqrt{ 2 - 2 \left( \min_{p, q \in \Cscr_1} \max_{ \hat{p}, \hat{q} \in \Cscr_2} \ip{ Up q^T V^T }{ U_r \hat{p} \hat{q}^T V_r^T } \right) } \nonumber \\
    & = \sqrt{ 2 - 2 \left( \min_{p, q \in \Cscr_1} \max_{ \hat{p}, \hat{q} \in \Cscr_2} \left( \sum_{i=1}^r p_i \hat{p}_i \right) \left( \sum_{i=1}^r q_i \hat{q}_i \right) \right) } \nonumber \\
    & = \sqrt{ 2 - 2 \left(\min_{p, q \in \Cscr_1} \| p_{1:r} \|_2 \| q_{1:r} \|_2 \right) }. \label{eq:rho_exp} 
  \end{align}
  Now we consider the subproblem in \eqref{eq:rho_exp}: 
  \begin{align}
    \min_{p,q}\enspace & \| p_{1:r} \|_2 \| q_{1:r} \|_2 \label{eq:reduced_p1} \tag{P$_{\mathrm{sub}}$} \\
    \text{subject to} & \sum_{i=1}^{k} \sigma_i p_i q_i \geq \sigma_1 - \epsilon_i,~ \|p\|_2 = \|q\|_2 = 1, ~p,q \in \mR^{ k }. \nonumber
  \end{align}
  If $p^*$ and $q^*$ is a solution of the problem \eqref{eq:reduced_p1}, then it's easy to verify that 
  \begin{align*}
    & \tilde{p} = \left[ \| p^*_{1:r} \|_2, 0,\ldots, \|p^*_{r+1:k}\|_2, 0, \ldots, 0 \right] \\ 
    \text{and} ~& \tilde{q} = \left[ \| q^*_{1:r} \|_2, 0,\ldots, \|q^*_{r+1:k}\|_2, 0, \ldots, 0 \right]
  \end{align*}
  is also a valid solution. Therefore there must exist solution $p^*, q^*$ such that $p_i = q_i = 0~\forall i \notin \{1, r+1\}$, that is only $p^*_1, q^*_1$ and $p^*_{r+1}$ and $q^*_{r+1}$ are greater or equal than 0. This allow us to further reduce the problem to
  \begin{align*}
    \min_{ p_1, q_1, p_{r+1}, q_{r+1} } &~  p_1 q_1 \\
    \text{subject to} & \sigma_1 p_1 q_1 + \sigma_{r+1} p_{r+1} q_{r+1} \geq \sigma_1 - \epsilon_i, \\
    & p_1^2 + p_{r+1}^2 = q_1^2 + q_{r+1}^2 = 1, ~p_1, q_1, p_{r+1}, q_{r+1} \geq 0. 
  \end{align*}
  It's easy to verify that when $\sigma_1 - \sigma_{r+1} \geq \epsilon_i$, the above problem attains solution at 
  \[
    p_1 = q_1 = \sqrt{\frac{\sigma_1 - \sigma_{r+1} - \epsilon_i}{ \sigma_1 - \sigma_{r+1} } } \text{and} p_{r+1} = q_{r+1} = \sqrt{1 - p_1^2}.
  \]
  When $\sigma_1 - \sigma_{r+1} < \epsilon_i$, the solution is simply $p_1 = q_1 = 0, p_{r+1}= q_{r+1} = 1$.
  Therefore the optimal value of \eqref{eq:reduced_p1} is $\max\{1 - \epsilon_i / ( \sigma_1 - \sigma_{r+1} ), 0\}$, plug this into \cref{eq:rho_exp} and we finish the proof. 
\end{proof}

\section{Proof of \Cref{prop:spectral}}

We describe the following lemma before proceeding to the proof of \Cref{prop:spectral}. 
\begin{lemma}[Hausdorff error bound] \label{lemma:dist_to_sol}
  Given $\widehat{\Ascr} \subseteq \Ascr$, there exists $ X \in \cone ( \widehat{\Ascr} ) $ such that 
  \begin{equation}
          \| X - X^* \|_F \leq \dist( \suppa(X^*), \widehat{\Ascr} )\cdot \sqrt{| \suppa(X^*) |} \cdot \| X^* \|_F.  \label{eq:dist_to_sol}
  \end{equation}
\end{lemma}

\begin{proof}
  Let $X^* = \sum_{a \in \suppa(X^*)} c_a a, c_a > 0$. By the definition of the one-sided Hausdorff distance $\dist(\cdot,\cdot)$, for any $a \in \suppa(X^*)$, there exist a corresponding $\hat{a} \in \widehat{\Ascr}$ such that 
  \[ 
    \| \hat{a} - a \|_F \leq \dist( \suppa(X^*), \widehat{\Ascr} ).
  \]
  Let $\hat{X} = \sum_{a \in \suppa(X^*)} c_a \hat{a} $, then it's straighforward to verify that $\hat{X} \in \cone ( \widehat{\Ascr} )$ and 
  \[
    \| X - X^* \|_F \leq \dist( \suppa(X^*), \widehat{\Ascr} ) \sum_{a \in \suppa(X^*)} c_a \overset{\rm{(i)}}{\leq} \dist( \suppa(X^*), \widehat{\Ascr} ) \sqrt{| \suppa(X^*) |} \| X^* \|_F,
  \]
  where (i) follows from the orthonormal decomposition $x^* = \sum_{a \in \suppa(X^*)} c_a a, c_a > 0$ and $\| X^* \|_F^2 = \sum c_a^2$ when our atomic set is the set of rank-one matrices.
\end{proof}

\begin{proof}[Proof of \Cref{prop:spectral}]
  By \Cref{lemma:dist_to_sol}, we know that there exist $\tilde{X}$ satisfies \cref{eq:dist_to_sol}.
  Then by the $L$-smoothness of $f$,
  \begin{align}
      f( b - M \tilde{X} ) &~\leq~ f( b - M x^* ) + \langle \nabla f( b - M X^* ), M( X^* - \tilde{X} ) \rangle + \frac{L}{2} \| M( X^* - \tilde{X} ) \|_F^2  \nonumber \\
      &~\leq~ f( b - M X^* ) + \| \nabla f( b - M X^* ) \|_F \| M( X^* - \tilde{X} ) \|_F + \frac{L}{2} \| M( X^* - \tilde{X} ) \|_F^2. \label{eq:errorBoundIntermediate1}
  \end{align}
  By the smoothness and convexity of $f$, we further have
  \begin{align*}
      \| \nabla f( b - M X^* ) - \nabla f(0) \|_F^2 \leq 2L ( f(b - M X^*) - f(0) ).
  \end{align*}
  Note we assume $f(0)=0$ and $\nabla f(0) = 0$, the above reduces to $\| \nabla f( b - M X^* ) \|_F \leq \sqrt{2 L \alpha}$. Combining with \cref{eq:errorBoundIntermediate1}, we obtain
  \begin{align*}
      &f( b - M \tilde{X} )    \\
      \leq~& f( b - M X^* ) + \sqrt{2 L \alpha} \| M( X^* - \tilde{X} ) \|_F + \frac{L}{2} \| M( X^* - \tilde{X} ) \|_F^2 \\
      \leq~& f( b - M X^* ) + \sqrt{2 L \alpha} \| M \| \dist( \suppa(X^*), \widehat{\Ascr} )\cdot \sqrt{| \suppa(X^*) |} \cdot \| X^* \|_F  + \frac{L \|M\|^2}{2} \dist( \suppa(X^*), \widehat{\Ascr} )^2 | \suppa(X^*) | \| X^* \|_F^2,
  \end{align*}
  where the last inequality is by \Cref{lemma:dist_to_sol}. Combining the above with \Cref{thm:svd_approx_score} leads to the desired result.
\end{proof}

\end{document}